\documentclass[10pt]{amsart}
\usepackage{amssymb, amsmath}
\usepackage{graphics}
\usepackage{latexsym}
\usepackage{amsmath}
\usepackage{amssymb,amsthm,amsfonts}
\usepackage{amscd}
\usepackage[arrow, matrix, curve]{xy}
\usepackage{syntonly}
\usepackage{yfonts}
\usepackage{tikz-cd}
\usepackage{filecontents}
\usepackage{mathtools}
\usepackage{stmaryrd}
\usepackage{mathtools}
\usepackage{MnSymbol}
\DeclareMathOperator{\aut}{Aut}

 \DeclareMathOperator{\jac}{Jac}
 \DeclareMathOperator{\nm}{Nm}

\DeclareMathOperator{\irr}{Irr}
\DeclareMathOperator{\im}{Im}
\DeclareMathOperator{\alb}{Alb}
\DeclareMathOperator{\en}{End}
\theoremstyle{plain}
\newtheorem{thm}{Theorem}[section]
\newtheorem{theorem}[thm]{Theorem}

\newtheorem{proposition}[thm]{Proposition}

\theoremstyle{definition}

\newtheorem{remark}[thm]{Remark}

\newtheorem{definition}[thm]{Definition}

\numberwithin{equation}{thm}
\newcommand{\rk}{{\rm rank}}

\newcommand{\Hom}{{\rm Hom}}

\begin{document}
\title{Rational points on abelian varieties over function fields and Prym varieties}
\author{Abolfazl Mohajer}
\address{Johannes Gutenberg Universit\"at Mainz, Institut f\"ur Mathematik, Staudingerweg 9, 55099 Mainz, Germany}
\maketitle 
\begin{abstract}
In this paper, using a generalization of the notion of Prym variety for covers of quasi-projective varieties, we prove a structure theorem for the Mordell-Weil group of the abelian varieties over function fields that are twists of Abelian varieties by Galois covers of irreducible quasi-projective varieties. In particular, the resutls we obtain contribute in the construction of Jacobians (of covers of the projective line) of high rank.
\end{abstract}	
\section{introduction}
Let $A$ be an abelian variety over a field $k$ of characteristic $\geq 0$. The Mordell-Weil theorem asserts that the set of $k$-rational points $A(k)$ on $A$ is a finitely generated abelian group. The rank of this abelian group is referred to as the rank
of the abelian variety $A$ with notation $\rk(A(k))$. One interesting problem in arithmetic geometry is to find abelian varieties with arbitrary large ranks. \par The present paper is a generalization of the works of Hazama and Salami, respectively in
\cite{Haz} and \cite{Sal} to arbitrary Galois coverings. More precisely, let $n\in \mathbb{N}$ be a natural number such that the characteristic of the field $k$ in the previous paragraph does not divide $n$. Suppose there is an embedding $G\hookrightarrow \aut(A)$ with $G$ a finite group with $|G|=n$. Further, let $X,Y$ be irreducible
quasi-projective varieties over $k$ with function fields $F,K$ respectively. We first define the notion of \emph{Prym variety} associated to a $G$-Galois covering $f:X\to
Y$. We will show that this is a natural generalization of the notion of Prym variety in the case of coverings of curves. The group $G$ is well-known to be also the Galois group of the Galois field extension $F/K$. By the results of \cite{BS} and \cite{Haz1},
the twist of $A$ by the extension $F/K$ is equivalent to a twist by the 1-cocyle $a=(a_g)\in Z^1(G,\aut(A))$ given by $a_g=g$, where $g$ is viewed both as a group element and as an automorphism of $A$ corresponding to $g\in G$ (in other words we identify $g$ with its image $g\in G\hookrightarrow \aut(A)$). We prove a theorem which gives an isomorphism for $A_a(K)$ in terms of the Prym variety introduced in section 2, see Theorem \ref{rational points}. Then we consider the $n$-times product $\prod_i f$ of a $G$-cover $f:X\to Y$ with itself. When $G$ is abelian, we prove a result which describes the Prym variety of the self product of $G$-covers of varieties in terms of the Prym variety of the cover. We apply this to the case of an abelian cover $C\to \mathbb{P}^1$. More
precisely we prove that $\rk(\jac(C)_a(K))$ can be made arbitrarily large by taking $n$ large enough.
\section{the Prym variety of a Galois covering}
Let $V$ be an irreducible quasi-projective variety over a field $k$. The \emph{Albanese variety} of $V$ is by definition the initial object for the morphisms from $V$ to abelian varieties. In particular, if $V$ is moreover non-singular over a field of characteristic zero (the classical case) it is the abelian variety,
\begin{equation}\label{Def albanese}
\alb(V)=H^0(V,\Omega^1_V)^*/H_1(V,\mathbb{Z})
\end{equation}
Now let $X,Y$ be irreducible quasi-projective varieties with $f:X\to Y$ a Galois covering with the Galois group $G$. Throughout this paper, we consider the situation where there is a finite group $G$ with $|G|=n$ that has a $G$-linearized action on $X$ such that $Y\coloneqq X/G$ and $f:X\to Y$ is the  quotient map. The universal property of the albanese variety implies that there is an induced action of $G$ on $\alb(X)$. We denote by ${(\alb(X)^G)}^0$ the largest abelian subvariety of $\alb(X)$ fixed (pointwise) under this action. Equivalently, ${(\alb(X)^G)}^0$ is the connected component containing the identity of the subvariety $\alb(X)^G$ of fixed points of $\alb(X)$ under the action of $G$. In characteristic zero the induced action of $G$ on $\alb(X)$ can be described through its action on $H^0(X,\Omega^1_X)$. With respect to this action, we write
\begin{align} \label{plus minus}
H^0(X,\Omega^1_X)^+=H^0(X,\Omega^1_X)^G(\cong
H^0(Y,\Omega^1_Y)) \text{ and }\nonumber\\
H^0(X,\Omega^1_X)^-= H^0(X,\Omega^1_X)/H^0(X,\Omega^1_X)^+=
\bigoplus\limits_{\chi\in\irr(G)\setminus\{1\}}H^0(\widetilde{C},\omega_{\widetilde{C}})^{\chi}
\end{align}
Similarly, one defines $H_1(X,\mathbb{Z})^-$. Here $\irr(G)$ is
the set of irreducible characters of $G$. \par Let $g\in G$ be an
element of the group $G$. We will denote also by $g$ the
automorphism of $\alb(X)$ induced by $g$. Using this notation and
following that of \cite{RR}, Prop 3.1 for the case of curves, we write
$\nm G:\alb(X)\to\alb(X)$ for the \emph{norm endomorphism} of
$\alb(X)$ given by $\nm G:=\sum_{g\in G}g$.\par Now we are ready
to define the notion of Prym variety for the covering $f$.  
\begin{definition}\label{prymdef}
Let $f:X\to Y$ be $G$-Galois covgering of irreducible quasi-projective varieties over a field $k$. The Prym variety $P(X/Y)$ associated with the covering $f$ is defined to be
\begin{equation}
P(X/Y):=\frac{\alb(X)}{{(\alb(X)^G)}^0}
\end{equation}
In particular $P(X/Y)$ is the complementary abelian subvariety in $\alb(X)$ of (the abelian subvariety) ${(\alb(X)^G)}^0$ .
\end{definition}
For the classical case of Prym varieties of double coverings of curves we refer to \cite{B}. The following result gives an equivalent description of the Prym variety and shows furthermore that in coincides with the Prym variety of covers of curves (see \cite{B}, \cite{LO}, \cite{M} and \cite{RR}) up to isogeny.
\begin{proposition}\label{prym properties}
With the above notation, $P(X/Y)=\frac{\alb(X)}{\im\nm G}$. Furthermore it is isogeneous to the abelian variety ${(\ker\nm G)}^0$. In particular, if the varieties are defined over a field of characteristic zero, then $P(X/Y)$ is isogeneous to the abelian variety $\alb(X)^-={(H^0(X,\Omega^1_X)^-)}^*/H_1(X,\mathbb{Z})^-$.
\end{proposition}
\begin{proof}
The induced linear action on the tangent space of $\alb(X)$ at the origin gives that $\dim\alb(X)=\dim{(\ker\nm G)}^0+\dim{(\alb(X)^G)}^0$. Let $Q\in{(\alb(X)^G)}^0\cap{(\ker\nm G)}^0$. It follows that $nQ=0$ in $\alb(X)$, therefore
${(\alb(X)^G)}^0\cap{(\ker\nm G)}^0\subseteq\alb(X)[n]$. In particular, this intersection is finite so that $\alb(X)$ is isogenous to the product ${(\alb(X)^G)}^0\times{(\ker\nm G)}^0$. Since $(\sum_{g\in G}g)\cdotp h=\sum_{g\in G}(g\cdotp h)=\sum_{g\in G}g$ for every $h\in G$, it follows that $\im\nm G\subseteq {(\alb(X)^G)}^0$. On the other hand the above isogeny decomposition for $\alb(X)$ implies that $\dim\im\nm G=\dim\alb(X)-\dim{(\ker\nm G)}^0=\dim{(\alb(X)^G)}^0$. Hence $\im\nm G={(\alb(X)^G)}^0$. This completes the proof of the claims.
\end{proof}
\begin{remark}\label{prym for curves}
For curves, the albanese variety coincides with the Jacobian variety. Hence if $X$ and $Y$ are curves, the Prym variety $P(X/Y)$ coincides with the Prym variety for covers of curves, see \cite{LO} and \cite{M}. In this case there are two fundamental homomorphisms: the norm homomorphism $\nm_f:\jac(X)\to\jac(Y)$ and the pull-back homomorphism $f^*:\jac(Y)\to\jac(X)$ and it holds that: $\nm G=f^*\circ\nm_f$, so that $P(X/Y)={(\ker\nm_f)}^0={(\ker\nm G)}^0$, see \cite{RR}, Prop 3.1.
\end{remark}
Recall from \cite{BS} and \cite{Haz1} that the twist of $A$ by the extension $F/K$ is equivalent to a twist by the 1-cocyle $a=(a_g)\in Z^1(G,\aut(A))$ given by $a_g=g$. We denote this twist by $A_a$.
\begin{theorem}\label{rational points}
Assume that there exists a $k$-rational point $x_0\in X(k)$. Then there is an isomorphism of abelian groups
\begin{equation}\label{rational isom}
A_a(K)\cong\Hom_k(P(X/Y),A)\oplus A[n](k)
\end{equation}
If in addition $P(X/Y)$ is $k$-isogenous to $A^n\times B$, where $n\in\mathbb{N}$ and $B$ is an abelian variety over $k$ with $\dim(B)=0$ or $\dim(B)>\dim(A)$ and $B$ does not have any simple components $k$-isogenous to $A$, then $\rk(A_a(K))\geq n\cdotp\rk(\en_k(A))$.
\end{theorem}
\begin{proof}
Note that the universal property of the Albanese variety together
with the fact that a regular morphism of abelian varieties is
obtained by a homomorphism followed by a translation, (see for
example \cite{LO},Prop 1.2.1) shows that
\begin{equation}\label{maps from albanese}
A(K)\cong\Hom_k(\alb(X),A)\oplus A(k)
\end{equation}
We assume that the Albanese map $\iota_X:X\to\alb(X)$ satisfies
$\iota_X(x_0)=0$ so that it is defined over $k$. Recall that \cite{Haz1}, Prop 1.1 shows that 
\[A_a(K)\cong\{P\in A(F)\mid a_g\cdotp ^g(P)=P \}\]
This implies in our particular case that for any $g\in G$, viewed as remarked earlier also
as an automorphism of $\alb(X)$ we have that $(\alpha,Q)\in A_a(K)$ if and only if $g(\alpha\circ g,Q)=(\alpha,Q)$ or equivalently $(\alpha\circ g,Q)=g^{-1}(\alpha,Q)$. But this is the case if and only if $\alpha$ annihilates $\im\nm G$
and $Q\in A[n](k)$. Proposition \ref{prym properties} then shows
that it must actually its lie in $P(X/Y)$, so we obtain the
claimed isomorphism in \ref{rational isom}.\par If moreover
$P(X/Y)$ is isogenous to $A^n\times B$ with $n,A$ and $B$ as in
the statement of proposition, then
\begin{align*}
A_a(K)\cong\Hom_k(P(X/Y),A)\oplus A[n](k)\\
\cong \Hom_k(A^n\times B,A)\oplus A[n](k)\\
\cong \Hom_k(A^n,A)\oplus\Hom_k(B,A)\oplus A[n](k)\\
\cong \en_k(A)^n\oplus\Hom_k(B,A)\oplus A[n](k).
\end{align*}
Which implies that as $\mathbb{Z}$-modules, it holds that
$\rk(A_a(K))\geq n\cdotp\rk(\en_k(U))$
\end{proof}
Given a $G$-Galois covering $f:X\to Y$, one can form the $n$-times self product $\prod_{i=1}^n f:\prod_{i=1}^n X\to\prod_{i=1}^n Y$ is a $\underbrace{G\times\cdots\times G}_{n-times}$- Galois covering. Suppose now that $G$ is abelian. Then the diagonal embedding $G\hookrightarrow \prod_i G:=G\times\cdots\times G$  
gives a subgroup of $\prod_i G$ isomorphic to $G$. We denote this subgroup by $\tilde{G}$. This gives rise to an intermediate Galois covering $f:\prod_{i=1}^n X\to(\prod_{i=1}^n X)/\tilde{G}$. Let us write $\mathcal{X}=\prod_{i=1}^n X$ and $\mathcal{Y}=(\prod_i
X)/\tilde{G}$. We are interested in the Prym variety $P(\mathcal{X}/\mathcal{Y})$. In fact we show,
\begin{proposition}\label{prym product}
With the above notation, there is an isogeny
\begin{equation}
P(\mathcal{X}/\mathcal{Y})\sim_k \prod_i P(X_i/ Y_i)
\end{equation}
\end{proposition}
\begin{proof}
It suffices to treat only the case $n=2$. The general case follows by an induction argument. So suppose $n=2$ and denote the Galois group of the cover $\mathcal{X}/\mathcal{Y}$ by $\tilde{G}(\cong G)$. By Proposition \ref{prym properties} it suffices to show that there is a $k$-isogeny ${(\ker\nm\tilde{G})}^0\sim_k{(\ker\nm G)}^0\times {(\ker\nm G)}^0$. In fact we show that there is an isomorphism between these abelian varieties. Notice that there is an isomorphism
\begin{equation}
\beta:\alb(X_1)\times\alb(X_2)\xrightarrow{\sim}\alb(\mathcal{X}).
\end{equation}
The isomorphism $\beta$ is given as follows: Let $j_i:X_i\to
\mathcal{X}$, for $i=1,2$ be the natural inclusions. Then
$\beta=\tilde{j_1}+\tilde{j_2}$, where $\tilde{j_i}$ denotes the
induced homomorphism $\alb(X_i)\to\alb(\mathcal{X})$. This
isomorphism is compatible with the action of $\tilde{G}$, namely,
there is the following commutative diagram. From this, one deduces
the isomorphism $\ker\nm\tilde{G}\xrightarrow{\sim}\ker\nm G\times
\ker\nm G$ which implies the desired isomorphism.
\end{proof}
Consider an abelian cover $C\to\mathbb{P}^1$ with Galois group $G$. Consider the product $\mathcal{C}_n=\prod_{i=1}^n C$, i.e., the product of $n$ copies of the same abelian cover in the above and let $\tilde{G}$ be the image of $G$ under the diagonal embedding $G\hookrightarrow\prod_{i=1}^n G$ as above. Set $\mathcal{D}_n=\mathcal{C}_n/\tilde{G}$. By Proposition \ref{prym product}, we have that
\begin{equation}
P(\mathcal{C}_n/\mathcal{D}_n)=\prod_i P(C/\mathbb{P}^1)
\end{equation}
By Remark \ref{prym for curves}, $P(C/\mathbb{P}^1)={(\ker\nm_f)}^0$. However as
$\jac(\mathbb{P}^1)=0$, it follows that $P(C/\mathbb{P}^1)=\jac(C)$. Now Proposition \ref{prym product} gives that
\begin{equation}\label{product for P1}
P(\mathcal{C}_n/\mathcal{D}_n)=\prod_i
P(C/\mathbb{P}^1)=(\jac(C))^n.
\end{equation}
Note that the function field $K(C)$ of $C$ is generated over the function field $K(\mathbb{P}^1)=K(z)$  of $\mathbb{P}^1$ by taking roots of (transcendental) elements of $K(C)$, i.e., it is of the form $K(z)(x_1^{1/m},\dots, x_r^{1/m})$. Then the function
field $\mathcal{L}_n$ of $\mathcal{C}_n$ is $K(z)(x_{i1}^{1/m},\dots, x_{ir}^{1/m}), i=1,\dots, n$. Let $K=k(\mathcal{D}_n)$ be the function field of $\mathcal{D}_n$.\par We define the 1-cocyle $Z^1(G,\aut(C))$ by $a_g=g$. Let $\jac(C)_a$ be the twist corresponding to this 1-cocyle. By applying \ref{rational isom} and \ref{product for P1}, it follows that
\begin{align*}
\jac(C)_a(K)\cong\Hom_k(P(\mathcal{C}_n/\mathcal{D}_n),\jac(C))\oplus \jac(C)[n](k)\\
\cong \Hom_k((\jac(C))^n,\jac(C))\oplus \jac(C)[n](k)\\
\cong \en_k(\jac(C))^n\oplus \jac(C)[n](k).
\end{align*}
So that $\rk(\jac(C)_a(K))\geq n\cdotp \rk(\en_k(\jac(C)))$.

\end{document}